\documentclass[12pt]{article}
\usepackage[utf8]{inputenc}
\usepackage[margin=1in]{geometry}
\usepackage{amsmath}
\usepackage{amssymb}
\usepackage{amsthm}
\usepackage{mathtools}
\usepackage[english]{babel}
\usepackage{booktabs}
\usepackage{fancyhdr}
\usepackage{titling}
\usepackage[colorlinks=true, allcolors=blue]{hyperref}

\title{Cotangent Models of Nilpotent Orbit Closures}
\author{Boming Jia}
\date{}
\hypersetup{pdftitle={Cotangent Models of Nilpotent Orbit Closures},pdfauthor={Boming Jia},pdfsubject={Cotangent models of nilpotent orbit closures}}

\newtheorem{theorem}{Theorem}[section]
\newtheorem{proposition}[theorem]{Proposition}
\newtheorem{lemma}[theorem]{Lemma}
\theoremstyle{remark}

\theoremstyle{plain}

\DeclareMathOperator{\Spec}{Spec}
\DeclareMathOperator{\Sing}{Sing}
\DeclareMathOperator{\rk}{rk}
\DeclareMathOperator{\Sym}{Sym}
\DeclareMathOperator{\Mat}{Mat}
\DeclareMathOperator{\Ad}{Ad}
\newcommand{\C}{\mathbb C}
\newcommand{\Gm}{\mathbb G_m}
\newcommand{\PP}{\mathbb P}
\newcommand{\OO}{\mathcal O}
\newcommand{\aff}{\mathrm{aff}}

\begin{document}
\setlength{\droptitle}{-6em}
\maketitle
\vspace{-3em}

\begin{abstract}
Let $\mathcal O$ be a nonzero nilpotent orbit in a complex simple Lie
algebra $\mathfrak g$.
For $\mathfrak g$ of classical type, we classify the orbit closures
$\overline{\mathcal O}$ that are isomorphic to
$(T^*X)^{\mathrm{aff}}$ for a smooth quasi-affine variety $X$.
In types $G_2$, $F_4$, and $E_8$, we show that no nonzero nilpotent
orbit closure admits such a cotangent model.
\end{abstract}

\section{Introduction.}\label{sec:introduction}

Let $X$ be a smooth quasi-affine variety.
We write $X^{\aff}=\Spec\Gamma(X,\mathcal O_X)$ for its affinization.
The cotangent bundle $T^*X$ is again quasi-affine, and we write
$(T^*X)^{\aff}$ for its affinization.

\begin{theorem}\label{thm:classification}
Let $\mathcal O$ be a nonzero nilpotent orbit in a complex simple Lie
algebra of classical type.
There is a smooth quasi-affine variety $X$ with
\[
\overline{\mathcal O}\cong(T^*X)^{\aff}
\]
if and only if $\mathcal O$ occurs in the following table.
\[
\begin{array}{c@{\qquad}l@{\qquad}l}
\toprule
\text{type}&\text{nonzero orbits } \OO_\lambda&\text{range}\\
\midrule
A_{n-1}&\lambda=(2^r,1^{n-2r})
&1\leq r\leq\lfloor(n-2)/2\rfloor\\[2pt]
B_n&\lambda=(2^2,1^{2n-3})
&n\geq3\\[2pt]
C_n&\lambda=(2^r,1^{2n-2r})
&1\leq r\leq n-1,\quad n\geq2\\[2pt]
D_n&\lambda=(2^{2r},1^{2n-4r})
&1\leq r\leq\lfloor(n-2)/2\rfloor,\quad n\geq4\\[2pt]
\bottomrule
\end{array}
\]
Here $\lambda$ denotes the Jordan partition of the nilpotent orbit
$\OO_\lambda$.
\end{theorem}

For each nilpotent orbit in the table above, Fu and Liu construct an
explicit smooth quasi-affine variety $X$ in
\cite[Table 1 and Theorems 1.4 and 4.2]{FuLiu}.
The following table records its affinization $X^{\aff}$.
\[
\begin{array}{c@{\qquad}l@{\qquad}l}
\toprule
\text{type}&\text{examples of }X^{\aff}&\text{parameters}\\
\midrule
A_{n-1}&\{A\in\Mat_{p\times q}:\rk A\leq r\}
&p+q=n,\quad 1\leq r<p\leq q\\[2pt]
B_n&\{v\in\C^{2n-1}:(v,v)=0\}
&n\geq3\\[2pt]
C_n&\{A\in\Sym^2\C^n:\rk A\leq r\}
&1\leq r\leq n-1,\quad n\geq2\\[2pt]
D_n&\{A\in\Lambda^2\C^n:\rk A\leq2r\}
&1\leq r\leq\lfloor(n-2)/2\rfloor,\quad n\geq4\\[-1pt]
&\{v\in\C^{2n-2}:(v,v)=0\}
&r=1,\quad n\geq4\\
\bottomrule
\end{array}
\]

In types $G_2$, $F_4$, and $E_8$, no nonzero nilpotent orbit closure
admits a cotangent model.
In types $E_6$ and $E_7$, the only remaining candidates are the two
minimal orbit closures and the closure of $\mathcal O(2A_1)$ in type
$E_7$.

\vspace{1em}
\textbf{Acknowledgments.}
The author was supported by NSFC Grant No.~12225108 and the Shuimu Scholar
Program in Tsinghua University. The author used ChatGPT extensively during
the preparation of this work. He would like to thank the model GPT-5.6 Sol
for generating an earlier draft of this paper and collaborating with the
author throughout the revision process.

\section{Group actions on nilpotent orbit closures.}

Let $\mathfrak g$ be a complex simple Lie algebra, and let $G$ be its
adjoint group.
Regard $G$ as a subgroup of $\operatorname{PGL}(\mathfrak g)$ via the
adjoint representation.
Let $\mathcal O\subset\mathfrak g$ be a nonzero nilpotent orbit.

\begin{lemma}\label{lem:intrinsic-vertex}
The vertex $0$ is the unique point $y\in\overline{\mathcal O}$ such that
\[
\dim T_y\overline{\mathcal O}=\dim\mathfrak g.
\]
So every algebraic automorphism of $\overline{\mathcal O}$ fixes the vertex $0$.
\end{lemma}

\begin{proof}
The linear span of $\overline{\mathcal O}$ is a nonzero
$G$-stable subspace of the adjoint representation.
Since $\mathfrak g$ is simple, the adjoint representation is
irreducible. Thus $\overline{\mathcal O}$ spans $\mathfrak g$.
Since $\overline{\mathcal O}$ is a cone, its tangent space at the
vertex is its linear span. Therefore
\[
T_0\overline{\mathcal O}=\mathfrak g.
\]

Let $y\ne0$.
The Killing quadratic form $q(z)=\kappa(z,z)$ vanishes on the
nilpotent cone. Its differential at $y$ is the nonzero linear form
$dq_y=2\kappa(y,-)$.
Hence
\[
\dim T_y\overline{\mathcal O}<\dim\mathfrak g.
\]
Since an automorphism preserves the dimensions of tangent spaces, it
fixes $0$.
\end{proof}

\begin{lemma}\label{lem:minimal-preserved}
Let $H$ be a connected algebraic group acting on
$\overline{\mathcal O}$.
Then $H$ preserves the minimal nilpotent orbit closure
$\overline{\mathcal O}_{\min}$.
\end{lemma}

\begin{proof}
If $\mathcal O=\mathcal O_{\min}$, there is nothing to prove. Assume
that $\mathcal O$ is not minimal. The boundary is the singular locus:
\[
\Sing(\overline{\mathcal O})
=\overline{\mathcal O}\setminus\mathcal O.
\]
Both $H$ and $G$ preserve
$\Sing(\overline{\mathcal O})$, hence act on the set of irreducible
components of $\Sing(\overline{\mathcal O})$.
Since both groups are connected, they preserve every component.
Since $\overline{\mathcal O}$ contains only finitely many $G$-orbits,
each irreducible component of $\Sing(\overline{\mathcal O})$ contains
a dense $G$-orbit and is therefore a nilpotent orbit closure.

The closure of every nonzero non-minimal nilpotent orbit contains
$\overline{\mathcal O}_{\min}$ in its boundary.
Its singular locus therefore has a nonzero irreducible component.
Choose such a component and repeat the argument.
The dimension drops at each step, so the process terminates.
The last nonzero closure has boundary $\{0\}$ and is
$\overline{\mathcal O}_{\min}$.
\end{proof}

\begin{lemma}\label{lem:sp-pgl}
Assume that $\mathfrak g=\mathfrak{sp}(V)$ for a $2n$-dimensional
complex vector space $V$, with $n\geq2$.
We identify
\[
\mathfrak{sp}(V)\cong\Sym^2V.
\]
If $\mathbb P\overline{\mathcal O}$ is invariant under the natural
$\operatorname{PGL}(V)$-action on $\mathbb P(\Sym^2V)$, then
\[
\overline{\mathcal O}=\overline{\mathcal O}_{\min}.
\]
\end{lemma}

\begin{proof}
The variety $\overline{\mathcal O}$ is the affine cone over
$\mathbb P\overline{\mathcal O}$.
Hence every $\operatorname{GL}(V)$-lift of the projective
$\operatorname{PGL}(V)$-action preserves $\overline{\mathcal O}$ under
its action on $\Sym^2V$.
Its nonzero orbits are classified by rank.
Let $r$ be the maximal rank of an element of $\overline{\mathcal O}$,
and put
\[
D_r=\{q\in\Sym^2V:\rk q\leq r\}.
\]
The rank-$r$ orbit is contained in $\overline{\mathcal O}$, so its
closure $D_r$ is contained in $\overline{\mathcal O}$.
Since $r$ is maximal, the reverse inclusion also holds.
Thus $\overline{\mathcal O}=D_r$.

We claim that $r=1$.
Choose a symplectic two-plane $W=\langle e,f\rangle\subset V$ with
$\omega(e,f)=1$, and consider $q=e\odot f\in\Sym^2V$.
Then $\rk q=2$.
Under the standard isomorphism
\[
\Phi:\Sym^2V\longrightarrow\mathfrak{sp}(V),
\qquad
\Phi(u\odot v)(x)=\omega(u,x)v+\omega(v,x)u,
\]
we have $\Phi(q)(e)=-e$ and $\Phi(q)(f)=f$.
Thus $\Phi(q)$ is not nilpotent.
Hence $D_r$ is not contained in the nilpotent cone whenever $r\geq2$.

Since $\overline{\mathcal O}$ is a nonzero nilpotent orbit closure, we
must have $r=1$. The rank-one cone
$D_1=\{v^2:v\in V\}$ is precisely
$\overline{\mathcal O}_{\min}$.
\end{proof}

\begin{proposition}\label{prop:projective-stabilizer}
Assume that
$(\mathfrak g,\mathcal O)$ is not
$(\mathfrak{sp}_{2n},\mathcal O_{\min})$ for $n\geq2$.
Set
\[
\operatorname{Stab}_{\operatorname{PGL}(\mathfrak g)}
\bigl(\mathbb P\overline{\mathcal O}\bigr)
\coloneqq
\{g\in\operatorname{PGL}(\mathfrak g):
g\mathbb P\overline{\mathcal O}
=\mathbb P\overline{\mathcal O}\}.
\]
Then the identity component
\[
\operatorname{Stab}_{\operatorname{PGL}(\mathfrak g)}
\bigl(\mathbb P\overline{\mathcal O}\bigr)^\circ
=G.
\]
\end{proposition}

\begin{proof}
Set
\[
K=
\operatorname{Stab}_{\operatorname{PGL}(\mathfrak g)}
\bigl(\mathbb P\overline{\mathcal O}\bigr)^\circ.
\]
Since $G$ preserves $\mathbb P\overline{\mathcal O}$, we have
$G\subseteq K$. The inverse image of $K$ in
$\operatorname{GL}(\mathfrak g)$ is connected and preserves
$\overline{\mathcal O}$. By Lemma \ref{lem:minimal-preserved}, it also
preserves $\overline{\mathcal O}_{\min}$. Thus
\[
K\subseteq
K_{\min}\coloneqq
\operatorname{Stab}_{\operatorname{PGL}(\mathfrak g)}
\bigl(\mathbb P\overline{\mathcal O}_{\min}\bigr)^\circ.
\]

Let $P_\theta\subset G$ be the stabilizer of a highest-root line. Then
\[
\mathbb P\overline{\mathcal O}_{\min}\cong G/P_\theta.
\]
Restriction defines a homomorphism
\[
K_{\min}\longrightarrow\operatorname{Aut}(G/P_\theta).
\]
This homomorphism is injective. Indeed, let $[A]$ lie in its kernel,
with $A\in\operatorname{GL}(\mathfrak g)$. Every point of
$G/P_\theta$ is an eigenline of $A$. Since $G/P_\theta$ is
irreducible, it lies in the projectivization of one eigenspace. The
adjoint variety spans $\mathbb P\mathfrak g$, so this eigenspace is
$\mathfrak g$. Hence $A$ is scalar. Since $K_{\min}$ is connected,
its image lies in $\operatorname{Aut}(G/P_\theta)^\circ$.

By Demazure's theorem
\cite[pp.~181--182, Th\'eor\`eme~1]{Demazure},
\[
\operatorname{Aut}(G/P_\theta)^\circ=G
\]
except for $(C_n,P_1)$, $(B_n,P_n)$, and $(G_2,P_1)$. The adjoint
variety is $C_n/P_1$, $B_n/P_2$, or $G_2/P_2$ in types $C_n$, $B_n$,
or $G_2$, respectively. Thus only type $C$ occurs among Demazure's
exceptional pairs. The case $B_2=C_2$ is the same excluded
symplectic minimal-orbit case. Outside type $C$, we have
\[
G\subseteq K\subseteq K_{\min}\subseteq
\operatorname{Aut}(G/P_\theta)^\circ=G.
\]
Hence $K=G$.

Suppose that $\mathfrak g=\mathfrak{sp}(V)$, where $\dim V=2n$ and
$n\geq2$. Under $\mathfrak{sp}(V)\cong\Sym^2V$, the minimal
projective orbit is the Veronese variety
\[
\nu_2(\mathbb P V)\subset\mathbb P(\Sym^2V),
\]
and $K_{\min}=\operatorname{PGL}(V)$. Therefore
\[
\operatorname{PSp}(V)\subseteq K\subseteq\operatorname{PGL}(V).
\]
Since
\[
\mathfrak{sl}(V)=\mathfrak{sp}(V)\oplus\Lambda^2_0V,
\qquad
\Lambda^2_0V=\ker\bigl(\Lambda^2V\xrightarrow{\,\omega\,}\C\bigr),
\]
and $\Lambda^2_0V$ is an irreducible $\operatorname{Sp}(V)$-module,
the Lie algebra of $K$ is either $\mathfrak{sp}(V)$ or
$\mathfrak{sl}(V)$. Since $K$ is connected, it is
$\operatorname{PSp}(V)$ or $\operatorname{PGL}(V)$. The second case
would make $\mathbb P\overline{\mathcal O}$ invariant under
$\operatorname{PGL}(V)$. By Lemma \ref{lem:sp-pgl}, this would imply
$\mathcal O=\mathcal O_{\min}$, contrary to the hypothesis. Hence
$K=\operatorname{PSp}(V)=G$.
\end{proof}

\section{Fiber dilation and fixed-point spaces.}\label{sec:fiber}

Let $\mathcal O$ be a nonzero nilpotent orbit of a complex simple Lie
algebra $\mathfrak g$. Suppose that
\[
\overline{\mathcal O}\cong(T^*X)^{\aff}
\]
for a smooth quasi-affine variety $X$.
The following lemma is well-known. We include the proof for
completeness.

\begin{lemma}\label{lem:small-boundary}
Let $Z$ be a smooth quasi-affine variety. If
$\Gamma(Z,\mathcal O_Z)$ is finitely generated, then the canonical map
\[
Z\longrightarrow Z^{\aff}
\]
is an open immersion and
\[
\operatorname{codim}_{Z^{\aff}}(Z^{\aff}\setminus Z)\geq2.
\]
\end{lemma}

\begin{proof}
The canonical map identifies $Z$ with an open subvariety of
$Z^{\aff}$ by \cite[Lemma 28.19.4]{Stacks}.
Since $Z$ is smooth, it is normal. Hence
$\Gamma(Z,\mathcal O_Z)$ is a normal domain, so $Z^{\aff}$ is an
irreducible normal affine variety.
By construction, $Z$ and $Z^{\aff}$ have the same ring of regular
functions. The codimension statement follows from
\cite[Theorem 4.2]{Grosshans}.
\end{proof}
Fiber dilation on $T^*X$ is
\[
t\cdot(x,\xi)=(x,t\xi).
\]
Under the assumed isomorphism, this induces a $\Gm$-action $\sigma$ on
$\overline{\mathcal O}$. By Lemma \ref{lem:intrinsic-vertex}, the
action $\sigma$ fixes the vertex $0$. Let
\[
\rho=d_0\sigma:\Gm\longrightarrow\operatorname{GL}(\mathfrak g)
\]
be its tangent representation at $0$. We use the tangent representation
in the sense of \cite[Section 5]{Jia}. By \cite[Lemma 4.5]{Jia}, the
total space $T^*X$ is smooth and quasi-affine. Put
\[
A=\Gamma(T^*X,\mathcal O_{T^*X}).
\]
The algebra $A$ is finitely generated because
$\Spec A\cong\overline{\mathcal O}$. By Lemma
\ref{lem:small-boundary}, $T^*X$ is a dense $\Gm$-stable open
subvariety of $(T^*X)^{\aff}$. In particular,
\[
\dim\mathcal O=2\dim X>0.
\]
Fiber dilation on $T^*X$ is therefore nontrivial, and hence so is the
induced action $\sigma$.
By \cite[Lemma 5.1]{Jia}, its tangent representation $\rho$ is
nontrivial.
Moreover, fiber dilation gives the nonnegative grading
\[
A=\bigoplus_{d\geq0}A_d,
\qquad
A_0=\Gamma(X,\mathcal O_X).
\]
Let $\mathfrak m\subset A$ be the maximal ideal of the vertex.
It is homogeneous because the vertex is fixed.
Choose homogeneous elements $f_1,\ldots,f_N\in\mathfrak m$ whose
classes $\bar f_1,\ldots,\bar f_N$ form a basis of
$\mathfrak m/\mathfrak m^2$, with $\deg f_i=d_i\geq0$, and let
$v_1,\ldots,v_N$ be the dual basis.
By the definition of the tangent representation in \cite[Section 5]{Jia},
\[
\rho(t)v_i=t^{d_i}v_i.
\]
Thus $\rho$ has only nonnegative weights.
By \cite[Lemma 5.2]{Jia}, the subgroup $\rho(\Gm)$ preserves
$\overline{\mathcal O}\subset\mathfrak g$.
By \cite[Lemma 5.4]{Jia}, the fixed-point subscheme of $\sigma$ is
$X^{\aff}=\Spec A_0$.
Let $x_0\in X^{\aff}$ correspond to the vertex, and let
$\mathfrak m_0=\mathfrak m\cap A_0$ be its maximal ideal. Put
\[
C_{x_0}(X^{\aff})
\coloneqq
\Spec\operatorname{gr}_{\mathfrak m_0}A_0,
\]
the tangent cone of $X^{\aff}$ at $x_0$.

\begin{lemma}\label{lem:fixed-tangent-cone}
There is an isomorphism of affine schemes
\[
C_{x_0}(X^{\aff})
\cong
\overline{\mathcal O}\cap
\mathfrak g^{\rho(\Gm)}.
\]
In particular,
\[
\dim\bigl(\overline{\mathcal O}\cap\mathfrak g^{\rho(\Gm)}\bigr)
=\dim X=\frac12\dim\mathcal O.
\]
\end{lemma}

\begin{proof}
Since the grading is nonnegative,
\[
(\mathfrak m^q)_0=\mathfrak m_0^q.
\]
Hence
\[
\operatorname{gr}_{\mathfrak m_0}A_0
\cong
\bigl(\operatorname{gr}_{\mathfrak m}A\bigr)_0.
\]
Combining this with \cite[Lemmas 5.2 and 5.4]{Jia}, we obtain an
isomorphism of affine schemes
\[
C_{x_0}(X^{\aff})
\cong
\overline{\mathcal O}\cap
\mathfrak g^{\rho(\Gm)}.
\]
Indeed, Lemma 5.2 identifies $C_0\overline{\mathcal O}$ with
$\overline{\mathcal O}$ as a $\rho$-stable closed subscheme of
$\mathfrak g$, while Lemma 5.4 identifies its fixed-point subscheme
with $\Spec((\operatorname{gr}_{\mathfrak m}A)_0)$.
The fixed-point subscheme of a $\rho$-stable closed subscheme of
$\mathfrak g$ is its scheme-theoretic intersection with
$\mathfrak g^{\rho(\Gm)}$.

The algebra $A_0=A^{\Gm}$ is finitely generated because $\Gm$ is
reductive. By Lemma \ref{lem:small-boundary}, $X$ is a dense open
subvariety of $X^{\aff}$. Since $A_0\subset A$ and $A$ is a domain,
$X^{\aff}$ is irreducible. The tangent cone has the same dimension as
the local ring at $x_0$. Therefore
\[
\dim C_{x_0}(X^{\aff})
=\dim X^{\aff}
=\dim X.
\]
Therefore
\[
\dim\bigl(\overline{\mathcal O}\cap\mathfrak g^{\rho(\Gm)}\bigr)
=\dim X
=\frac12\dim\mathcal O.
\qedhere
\]
\end{proof}

Let $G$ be the adjoint group of $\mathfrak g$. Fix a Cartan subalgebra
$\mathfrak h\subset\mathfrak g$, and let $\Phi$ be its root system.
Fix simple roots $\alpha_1,\ldots,\alpha_n$, and write the highest root as
\[
\theta=\sum_{i=1}^n m_i\alpha_i.
\]
Let $\omega_1^\vee,\ldots,\omega_n^\vee$ be the fundamental
coweights. For a nonempty set $I\subset\{1,\ldots,n\}$, put
\[
\mathfrak g_I=
\bigoplus_{\substack{\beta=\sum b_j\alpha_j\in\Phi\\
b_i=m_i\text{ for every }i\in I}}
\mathfrak g_\beta.
\]

\begin{lemma}\label{lem:top-root}
Assume that the image of $\rho$ in
$\operatorname{PGL}(\mathfrak g)$ is contained in the adjoint group
$G$.
Then, up to conjugation by $G$, there exist an integer $c>0$ and a
dominant cocharacter
\[
\mu=\sum_i a_i\omega_i^\vee,
\qquad a_i\geq0,
\]
such that
\[
\rho(t)=t^c\Ad(\mu(t)^{-1}),
\qquad
c=\langle\theta,\mu\rangle,
\qquad
\mathfrak g^{\rho(\Gm)}=\mathfrak g_I,
\]
where $I=\{i:a_i>0\}$ is nonempty.
\end{lemma}

\begin{proof}
This proof is almost identical to that of
\cite[Lemma 3.3]{Jia}.
Let $\overline\rho:\Gm\to\operatorname{PGL}(\mathfrak g)$ be the
projectivization of $\rho$.
Applying a suitable conjugation by an element of $G$, we may assume that
$\overline\rho$ is anti-dominant. Thus
\[
\overline\rho(t)=\Ad(\mu(t)^{-1})
\]
for a dominant cocharacter
\[
\mu=\sum_i a_i\omega_i^\vee,
\qquad
a_i\in\mathbb Z_{\geq0}.
\]
The map $\rho(t)\Ad(\mu(t))$ is scalar and depends homomorphically on
$t$. Hence these scalars form a character of $\Gm$, and for some
$c\in\mathbb Z$ we have
\[
\rho(t)=t^c\Ad(\mu(t)^{-1}).
\]
The weight of $\rho$ on $\mathfrak h$ is $c$, and its weight on
$\mathfrak g_\beta$ is $c-\langle\beta,\mu\rangle$. Since all these
weights are nonnegative and $\mu$ is dominant, we have
$c\geq\langle\theta,\mu\rangle\geq0$.

We first show that $c>0$. If $c=0$, then
$\langle\theta,\mu\rangle=0$. Every coefficient $m_i$ is positive and
every $a_i$ is nonnegative, so $\mu=0$. But $\rho$ is nontrivial, a
contradiction. Thus $\mathfrak h$ contains no fixed vector.
By Lemma \ref{lem:fixed-tangent-cone},
\[
\dim\bigl(\overline{\mathcal O}\cap
\mathfrak g^{\rho(\Gm)}\bigr)=\dim X>0.
\]
Thus $\mathfrak g^{\rho(\Gm)}\ne0$, so some root space
$\mathfrak g_\beta$ is fixed. Then
$\langle\beta,\mu\rangle=c>0$, so $\beta$ is a positive root.
It follows that
\[
c=\langle\beta,\mu\rangle
\leq\langle\theta,\mu\rangle\leq c,
\]
and hence $c=\langle\theta,\mu\rangle$.

Put $I=\{i:a_i>0\}$. The set $I$ is nonempty, since $\mu=0$ would
again imply $c=0$ and make $\rho$ trivial. We now prove the two
inclusions separately.

First take a root $\beta=\sum_i b_i\alpha_i$ for which
$\mathfrak g_\beta$ is fixed. Since $\theta$ is the highest root, we have
$b_i\leq m_i$ for every $i$. Moreover,
\[
0=\langle\theta-\beta,\mu\rangle
=\sum_i a_i(m_i-b_i).
\]
Every summand is nonnegative. Thus $b_i=m_i$ whenever $a_i>0$, or
equivalently, whenever $i\in I$. Since $\mathfrak h$ has no fixed
vector, this proves
$\mathfrak g^{\rho(\Gm)}\subseteq\mathfrak g_I$.

Conversely, take a root $\beta=\sum_i b_i\alpha_i$ such that
$b_i=m_i$ for every $i\in I$. If $i\notin I$, then $a_i=0$. Therefore
\[
\langle\theta-\beta,\mu\rangle
=\sum_i a_i(m_i-b_i)=0.
\]
Hence $\langle\beta,\mu\rangle=\langle\theta,\mu\rangle=c$, so
$\mathfrak g_\beta$ is fixed by $\rho$. This proves
$\mathfrak g_I\subseteq\mathfrak g^{\rho(\Gm)}$ and completes the proof.
\end{proof}

\begin{lemma}\label{lem:strict-bound}
We have
\[
\dim\mathfrak g^{\rho(\Gm)}\geq\dim X+1.
\]
\end{lemma}

\begin{proof}
By Lemma \ref{lem:fixed-tangent-cone},
\[
\dim\mathfrak g^{\rho(\Gm)}
\geq
\dim\bigl(\overline{\mathcal O}\cap\mathfrak g^{\rho(\Gm)}\bigr)
=\dim X.
\]
Suppose, to the contrary, that equality holds. By
\cite[Lemma 5.4]{Jia},
\[
T_{x_0}(X^{\aff})
=(T_0\overline{\mathcal O})^{\rho(\Gm)}
=\mathfrak g^{\rho(\Gm)}.
\]
By Lemma \ref{lem:small-boundary}, $X$ is a dense open subvariety of
$X^{\aff}$. Thus $\dim X^{\aff}=\dim X$, and $x_0$ is smooth in
$X^{\aff}$.

Put $U=(X^{\aff})_{\mathrm{reg}}$. Since $X$ is smooth and open in
$X^{\aff}$, we have $X\subset U$. By Lemma \ref{lem:small-boundary},
\[
\operatorname{codim}_U(U\setminus X)\geq2.
\]
Since $T^*U\to U$ is a vector bundle and $T^*X=T^*U|_X$, we have
\[
\operatorname{codim}_{T^*U}(T^*U\setminus T^*X)
=\operatorname{codim}_U(U\setminus X)\geq2.
\]
Since $T^*U$ is normal, regular functions extend across subsets of
codimension at least two. Hence
\[
\Gamma(T^*U,\mathcal O_{T^*U})
=
\Gamma(T^*X,\mathcal O_{T^*X})=A.
\]
Since $T^*U$ is smooth and quasi-affine, Lemma
\ref{lem:small-boundary} shows that the canonical map
\[
T^*U\longrightarrow\Spec A=\overline{\mathcal O}
\]
is an open immersion.

Let $z_0\in T^*_{x_0}U$ be the zero covector. We have
\[
\mathfrak m=\mathfrak m_0\oplus\bigoplus_{d>0}A_d.
\]
The restriction isomorphism
$\Gamma(T^*U,\mathcal O_{T^*U})\cong A$ is $\Gm$-equivariant. Hence
the grading on $A$ is also the fiber grading for $T^*U$.
Elements of $\mathfrak m_0$ vanish at $x_0$, while every
positive-degree function vanishes on the zero section. Thus $z_0$
maps to the vertex $0$. By Lemma \ref{lem:intrinsic-vertex},
\[
T_0\overline{\mathcal O}=\mathfrak g.
\]
Since $\overline{\mathcal O}$ is a proper subvariety of $\mathfrak g$,
the vertex $0$ is singular. An open immersion cannot send the smooth
point $z_0$ to a singular point. This is a contradiction.
\end{proof}

\section{Classification in classical types.}\label{sec:classical}

Let $\mathfrak g$ be of classical type, and let $V$ be its standard
representation. Suppose that
\[
\overline{\mathcal O}\cong(T^*X)^{\aff}
\]
for a smooth quasi-affine variety $X$. Let
$\rho:\Gm\to\operatorname{GL}(\mathfrak g)$ be the tangent
representation induced by fiber dilation.
By Lemmas \ref{lem:fixed-tangent-cone} and \ref{lem:strict-bound},
\begin{equation}\label{eq:fixed-dimensions}
\dim(\overline{\mathcal O}\cap \mathfrak g^{\rho(\Gm)})
=\frac12\dim\mathcal O,
\qquad
\dim \mathfrak g^{\rho(\Gm)}\geq\frac12\dim\mathcal O+1.
\end{equation}

Except in the minimal-orbit case for
$\mathfrak g\simeq\mathfrak{sp}_{2n}$ with $n\geq2$, Proposition
\ref{prop:projective-stabilizer} shows that the projective image of
$\rho$ lies in $G$. Lemma \ref{lem:top-root} therefore applies.
We use its notation $c$, $\mu$, and $I$, so that
\begin{equation}\label{eq:rho-classical}
\rho(t)=t^c\Ad(\mu(t)^{-1}),
\qquad
c=\langle\theta,\mu\rangle,
\qquad
\mathfrak g^{\rho(\Gm)}=\mathfrak g_I.
\end{equation}
After precomposing $\rho$ by a suitable map $t\mapsto t^d$ with $d>0$, we may
lift $\mu$ to a cocharacter $\widetilde{\mu}$ of the classical group
acting on $V$. This replaces $c$ and $\mu$ by $dc$ and $d\mu$ and does
not change $\mathfrak g^{\rho(\Gm)}$. So we still use the same
notations $\rho$, $c$, and $\mu$ after this replacement, and always
assume that $\mu$ lifts to $\widetilde{\mu}$.

\subsection{Type A}\label{sec:typeA}

\begin{proposition}\label{prop:typeA}
Let $\mathfrak g=\mathfrak{sl}(V)$, where $\dim V=n$, and let
$\mathcal O_\lambda$ be a nonzero nilpotent orbit. If
\[
\overline{\mathcal O_\lambda}\cong(T^*X)^{\aff}
\]
for a smooth quasi-affine variety $X$, then
there are integers $p,q,r$ with
\[
p+q=n,
\qquad
1\leq r<p\leq q,
\qquad
\lambda=(2^r,1^{n-2r}).
\]
\end{proposition}

\begin{proof}
Use the notation $\rho$, $c$, $\mu$, and $\widetilde{\mu}$ introduced
above. Write
\[
V=\bigoplus_{a\in\mathbb Z}V_a.
\]
Let $a_{\max}$ and $a_{\min}$ be the extreme weights, and put
\[
P=V_{a_{\max}},\qquad Q=V_{a_{\min}},\qquad
p_0=\dim P,\qquad q_0=\dim Q.
\]

The $\operatorname{Ad}(\widetilde{\mu})$-weight of
$\operatorname{Hom}(V_b,V_a)$ is $a-b$. Its largest value is therefore
$a_{\max}-a_{\min}$. Since $\mu$ is dominant, this value is
$\langle\theta,\mu\rangle$. Thus
\begin{equation}\label{eq:c-typeA}
c=a_{\max}-a_{\min}.
\end{equation}
The $\rho$-weight on $\operatorname{Hom}(V_b,V_a)$ is $c-a+b$.
By \eqref{eq:c-typeA}, it is zero exactly when
$(a,b)=(a_{\max},a_{\min})$. Hence
\[
\mathfrak g^{\rho(\Gm)}=\operatorname{Hom}(Q,P).
\]
Here each map is extended by zero on the other weight spaces. The
cocharacter $\mu$ is nontrivial, so $a_{\max}>a_{\min}$. In
particular, $P\cap Q=0$ and $p_0+q_0\leq n$.

Every $A\in\operatorname{Hom}(Q,P)$ satisfies $A^2=0$. If $A$ has
rank $t$, then its Jordan partition is
\[
\lambda_t=(2^t,1^{n-2t}).
\]
Let $r$ be the largest integer such that
\[
\lambda_r\leq\lambda,\qquad r\leq\min(p_0,q_0),
\]
where $\leq$ is the dominance order. The intersection in
\eqref{eq:fixed-dimensions} has positive dimension, so $r\geq1$.
The type $A$ closure order \cite{KP79} shows that its closed points are
precisely
\[
\{A\in\operatorname{Hom}(Q,P):\rk A\leq r\}.
\]
Indeed, an element of rank $t$ belongs to
$\overline{\mathcal O_\lambda}$ if and only if
$\lambda_t\leq\lambda$.

This determinantal variety has dimension $r(p_0+q_0-r)$. Moreover,
$\dim\mathcal O_{\lambda_r}=2r(n-r)$ \cite{KP79}. Therefore
\[
\begin{aligned}
\frac12\dim\mathcal O_\lambda
&=r(p_0+q_0-r)\\
&\leq r(n-r)
=\frac12\dim\mathcal O_{\lambda_r}
\leq\frac12\dim\mathcal O_\lambda.
\end{aligned}
\]
The first equality follows from \eqref{eq:fixed-dimensions}. The last
inequality follows from $\lambda_r\leq\lambda$. Thus equality holds
throughout. Since $r>0$, we obtain $p_0+q_0=n$. We also have
$\mathcal O_{\lambda_r}\subseteq\overline{\mathcal O_\lambda}$ and
the two orbits have the same dimension. Hence $\lambda=\lambda_r$.

If $r=\min(p_0,q_0)$, the rank condition above is vacuous. The
intersection would then have the same dimension as
$\mathfrak g^{\rho(\Gm)}$, contrary to
\eqref{eq:fixed-dimensions}. Therefore $r<\min(p_0,q_0)$. Set
\[
p=\min(p_0,q_0),\qquad q=\max(p_0,q_0).
\]
Then
\[
p+q=n,\qquad 1\leq r<p\leq q,\qquad
\lambda=(2^r,1^{n-2r}).\qedhere
\]
\end{proof}

\subsection{Type B}\label{sec:typeB}

\begin{proposition}\label{prop:typeB}
Let $\mathfrak g=\mathfrak{so}(V)$, where $\dim V=2n+1\geq7$, and
let $\mathcal O_\lambda$ be a nonzero nilpotent orbit. If
\[
\overline{\mathcal O_\lambda}\cong(T^*X)^{\aff}
\]
for a smooth quasi-affine variety $X$, then
\[
\lambda=(2^2,1^{2n-3}),
\qquad\text{i.e.}\qquad
\mathcal O_\lambda=\mathcal O_{\min}.
\]
\end{proposition}

\begin{proof}
Use the notation $\rho$, $c$, $\mu$, and $\widetilde{\mu}$ introduced
above. Put $N=2n+1$, and write
\[
V=\bigoplus_{a\in\mathbb Z}V_a.
\]
The symmetric form pairs $V_a$ with $V_{-a}$. Let $a_{\max}>0$ be the
largest weight, and put
\[
W=V_{a_{\max}},\qquad p=\dim W.
\]
Then $W$ is isotropic. We identify $\mathfrak{so}(V)$ with
$\Lambda^2V$, where $u\wedge v$ acts by
\[
x\longmapsto (u,x)v-(v,x)u.
\]

We first show that $p=1$. Suppose that $p\geq2$. Then
\[
c=2a_{\max},
\qquad
\mathfrak g^{\rho(\Gm)}=\Lambda^2W.
\]
Every element of $\Lambda^2W$ is square-zero. An element of rank $2t$
has partition
\[
\lambda_t=(2^{2t},1^{N-4t}).
\]
Let $r$ be the largest integer such that
\[
\lambda_r\leq\lambda,
\qquad
2r\leq p.
\]
By \eqref{eq:fixed-dimensions}, we have $r\geq1$. The orthogonal
closure order \cite{KP} shows that the closed points of
$\overline{\mathcal O_\lambda}\cap\mathfrak g^{\rho(\Gm)}$
are precisely the elements of $\Lambda^2W$ of rank at most $2r$.
Hence
\[
\begin{aligned}
\frac12\dim\mathcal O_\lambda
&=r(2p-2r-1)\\
&\leq r(N-2r-1)
=\frac12\dim\mathcal O_{\lambda_r}\\
&\leq\frac12\dim\mathcal O_\lambda.
\end{aligned}
\]
Thus all inequalities are equalities. Since $r>0$, we have $2p=N$,
contrary to the parity of $N$. Therefore $p=1$.

Let $b<a_{\max}$ be the largest weight with $V_b\ne0$. We claim that
$b\geq0$. Suppose that $b<0$. Then $-b$ is a weight. The choice of
$b$ forces $-b=a_{\max}$. Any weight smaller than $b$ would have a
negative larger than $a_{\max}$. Thus the only weights are
$a_{\max}$ and $-a_{\max}$, both with multiplicity one. Hence $N=2$,
a contradiction. Therefore $b\geq0$. Since $\Lambda^2W=0$,
\[
c=a_{\max}+b,
\qquad
\mathfrak g^{\rho(\Gm)}=W\otimes V_b.
\]

We next show that $b=0$. Suppose that $b>0$. Every nonzero element of
$W\otimes V_b$ is square-zero of rank two. Hence
\[
\mathfrak g^{\rho(\Gm)}
\subseteq\overline{\mathcal O}_{\min}
\subseteq\overline{\mathcal O_\lambda},
\]
contrary to \eqref{eq:fixed-dimensions}. Thus $b=0$.

Put $U=V_0$. Any negative weight other than $-a_{\max}$ would
have a positive opposite weight strictly between $0$ and $a_{\max}$.
Thus $U$ is nondegenerate, $\dim U=N-2$, and
\[
\mathfrak g^{\rho(\Gm)}=W\otimes U.
\]
Fix $0\ne w\in W$. Every element has the form $A=w\wedge u$, and
\[
A^2(x)=-(u,u)(w,x)w.
\]
For $u\ne0$, the map $A$ has rank two. If $(u,u)=0$, then $A^2=0$
and $A$ has partition $(2^2,1^{N-4})$. If $(u,u)\ne0$, then $A^2$
has rank one and $A^3=0$, so $A$ has partition $(3,1^{N-3})$.

We claim that every closed point $w\wedge u$ of
$\overline{\mathcal O_\lambda}\cap\mathfrak g^{\rho(\Gm)}$
satisfies $(u,u)=0$. Suppose otherwise. Then the intersection contains
an element with partition
\[
\lambda_3=(3,1^{N-3}).
\]
Since $\overline{\mathcal O_\lambda}$ is $\operatorname{SO}(V)$-stable,
\[
\mathcal O_{\lambda_3}\subseteq\overline{\mathcal O_\lambda},
\qquad
\lambda_3\leq\lambda.
\]
Every element of $\mathfrak g^{\rho(\Gm)}$ is zero or has partition
$(2^2,1^{N-4})$ or $\lambda_3$. The minimal orbit closure is contained
in every nonzero nilpotent orbit closure. Therefore
\[
\mathfrak g^{\rho(\Gm)}\subseteq\overline{\mathcal O_\lambda},
\]
contrary to \eqref{eq:fixed-dimensions}. This proves the claim.

Since
$\overline{\mathcal O}_{\min}\subseteq\overline{\mathcal O_\lambda}$,
the closed points of
$\overline{\mathcal O_\lambda}\cap\mathfrak g^{\rho(\Gm)}$ are
exactly
\[
\{w\wedge u:(u,u)=0\}.
\]
This cone has dimension $N-3$. Hence
\[
\frac12\dim\mathcal O_\lambda=N-3
=\frac12\dim\mathcal O_{\min}.
\]
Since $\mathcal O_{\min}\subseteq\overline{\mathcal O_\lambda}$ and
the two orbits have the same dimension,
\[
\mathcal O_\lambda=\mathcal O_{\min}.
\qedhere
\]
\end{proof}

\subsection{Type C}\label{sec:typeC}

\begin{proposition}\label{prop:typeC}
Let $\mathfrak g=\mathfrak{sp}(V)$, where $\dim V=2n$ and $n\geq2$,
and let $\mathcal O_\lambda$ be a nonzero nilpotent orbit. If
\[
\overline{\mathcal O_\lambda}\cong(T^*X)^{\aff}
\]
for a smooth quasi-affine variety $X$, then
\[
\lambda=(2^r,1^{2n-2r}),
\qquad
1\leq r\leq n-1.
\]
\end{proposition}

\begin{proof}
If $\mathcal O_\lambda=\mathcal O_{\min}$, the conclusion holds with
$r=1$. Assume from now on that
$\mathcal O_\lambda\ne\mathcal O_{\min}$. Then
Lemma \ref{lem:top-root} and \eqref{eq:rho-classical} apply. Use the
notation $\rho$, $c$, $\mu$, and $\widetilde{\mu}$ introduced above,
and write
\[
V=\bigoplus_{a\in\mathbb Z}V_a.
\]
The symplectic form pairs $V_a$ with $V_{-a}$. Let $a_{\max}>0$ be
the largest weight, and put
\[
W=V_{a_{\max}},\qquad p=\dim W.
\]
Then $W$ is isotropic and $p\leq n$.

We realize $\mathfrak{sp}(V)$ as $\Sym^2V$, where $u\odot v$ acts by
\[
x\longmapsto \omega(u,x)v+\omega(v,x)u.
\]
The $\operatorname{Ad}(\widetilde{\mu})$-weight on
$V_a\otimes V_b$ is $a+b$. Its largest value is $2a_{\max}$, which is
$\langle\theta,\mu\rangle$ because $\mu$ is dominant. Hence
\begin{equation}\label{eq:c-typeC}
c=2a_{\max}.
\end{equation}
The corresponding $\rho$-weight is $c-a-b$. By
\eqref{eq:c-typeC}, it is zero exactly when
$a=b=a_{\max}$. Therefore
\[
\mathfrak g^{\rho(\Gm)}=\Sym^2W.
\]
Every element of $\Sym^2W$ has image in $W$ and vanishes on $W$.
It is therefore square-zero. An element of rank $t$ has partition
\[
\lambda_t=(2^t,1^{2n-2t}).
\]

Let $r$ be the largest integer such that
\[
\lambda_r\leq\lambda,\qquad r\leq p.
\]
By \eqref{eq:fixed-dimensions}, we have $r\geq1$. The symplectic
closure order \cite{KP} shows that the closed points of
$\overline{\mathcal O_\lambda}\cap\mathfrak g^{\rho(\Gm)}$ are
precisely
\[
\{A\in\Sym^2W:\rk A\leq r\}.
\]
Indeed, a rank-$t$ element belongs to
$\overline{\mathcal O_\lambda}$ if and only if
$\lambda_t\leq\lambda$.

This symmetric determinantal variety has dimension
$r(2p-r+1)/2$, while
$\dim\mathcal O_{\lambda_r}=r(2n-r+1)$ \cite{KP}. Hence
\[
\begin{aligned}
\frac12\dim\mathcal O_\lambda
&=\frac{r(2p-r+1)}2\\
&\leq\frac{r(2n-r+1)}2
=\frac12\dim\mathcal O_{\lambda_r}
\leq\frac12\dim\mathcal O_\lambda.
\end{aligned}
\]
The first equality follows from \eqref{eq:fixed-dimensions}. The last
inequality follows from $\lambda_r\leq\lambda$. Thus equality holds
throughout. Since $r>0$, we obtain $p=n$. Moreover,
$\mathcal O_{\lambda_r}\subseteq\overline{\mathcal O_\lambda}$ and
the two orbits have the same dimension. Hence
$\lambda=\lambda_r$.

If $r=n$, the rank condition above is vacuous. The intersection would
then have the same dimension as $\mathfrak g^{\rho(\Gm)}$, contrary
to \eqref{eq:fixed-dimensions}. Therefore
\[
\lambda=(2^r,1^{2n-2r}),\qquad 1\leq r\leq n-1.\qedhere
\]
\end{proof}

\subsection{Type D}
\label{sec:typeD}

If $\lambda$ is very even, $\mathcal O_\lambda$ denotes
either of the two $\operatorname{SO}(V)$-orbits with Jordan partition
$\lambda$.

\begin{proposition}\label{prop:typeD}
Let $\mathfrak g=\mathfrak{so}(V)$, where $\dim V=2n\geq8$, and let
$\mathcal O_\lambda$ be a nonzero nilpotent orbit. If
\[
\overline{\mathcal O_\lambda}\cong(T^*X)^{\aff}
\]
for a smooth quasi-affine variety $X$, then
\[
\lambda=(2^{2r},1^{2n-4r}),
\qquad
1\leq r\leq\lfloor(n-2)/2\rfloor.
\]
\end{proposition}

\begin{proof}
Use the notation $\rho$, $c$, $\mu$, and $\widetilde{\mu}$ introduced
above. Put $N=2n$, and write
\[
V=\bigoplus_{a\in\mathbb Z}V_a.
\]
The symmetric form pairs $V_a$ with $V_{-a}$. Let $a_{\max}>0$ be the
largest weight, and put
\[
W=V_{a_{\max}},\qquad p=\dim W.
\]
Then $W$ is isotropic and $2p\leq N$. As in the proof of Proposition
\ref{prop:typeB}, we identify $\mathfrak{so}(V)$ with $\Lambda^2V$.

Suppose first that $p\geq2$. Then
for $u\in V_a$ and $v\in V_d$, the
$\operatorname{Ad}(\widetilde\mu)$-weight of $u\wedge v$ is $a+d$.
Since $\Lambda^2W\ne0$, the largest such weight is
$2a_{\max}$, and equality forces $a=d=a_{\max}$.
Since $c=\langle\theta,\mu\rangle$, it follows that
\[
c=2a_{\max},
\qquad
\mathfrak g^{\rho(\Gm)}=\Lambda^2W.
\]
Every element of $\Lambda^2W$ has image in $W$ and vanishes on $W$.
It is therefore square-zero. An element of rank $2t$ has partition
\[
\lambda_t=(2^{2t},1^{N-4t}).
\]

Put $m=\lfloor p/2\rfloor$, and for $0\leq t\leq m$ put
\[
R_t=\{A\in\Lambda^2W:\rk A=2t\}.
\]
Choose an isotropic subspace $W^-$ paired perfectly with $W$.
The subgroup $\operatorname{GL}(W)\subset\operatorname{SO}(V)$ acts
on $W^-$ by the inverse dual action and acts trivially on
$(W\oplus W^-)^\perp$. It preserves $\Lambda^2W$ and acts
transitively on each $R_t$. Hence each $R_t$ is either contained in
$\overline{\mathcal O_\lambda}$ or disjoint from it.

The intersection in \eqref{eq:fixed-dimensions} has positive
dimension, so it contains a nonzero closed point. Let $r$ be the
largest integer such that
$R_r\subseteq\overline{\mathcal O_\lambda}$. Then $r\geq1$.
The closure of $R_r$ in $\Lambda^2W$ is the Pfaffian variety
\[
D_r=\{A\in\Lambda^2W:\rk A\leq2r\}.
\]
Since $\overline{\mathcal O_\lambda}$ is closed, we have
$D_r\subseteq\overline{\mathcal O_\lambda}$. Conversely, let
$A$ be a closed point of
$\overline{\mathcal O_\lambda}\cap\Lambda^2W$, and write
$\rk A=2t$. Then $R_t$ meets $\overline{\mathcal O_\lambda}$, so the
transitivity above shows that
$R_t\subseteq\overline{\mathcal O_\lambda}$. The maximality of $r$
implies $t\leq r$, and hence $A\in D_r$. Thus the closed points of
$\overline{\mathcal O_\lambda}\cap\mathfrak g^{\rho(\Gm)}$ are
exactly the points of $D_r$. In particular,
\[
\begin{aligned}
\dim\bigl(\overline{\mathcal O_\lambda}
\cap\mathfrak g^{\rho(\Gm)}\bigr)
&=\dim D_r\\
&=2r(p-2r)+\binom{2r}{2}\\
&=r(2p-2r-1).
\end{aligned}
\]
Here the middle expression comes from choosing the $2r$-dimensional
image of a rank-$2r$ element and then choosing a nondegenerate
alternating tensor on that image.

Choose $A_r\in R_r$, and let
$\mathcal P_r=\operatorname{SO}(V)\cdot A_r$. This orbit has Jordan
partition $\lambda_r$. If $\lambda_r$ is very even, this definition
selects the orbit containing $A_r$ from the two possible orbits. Since
$A_r\in\overline{\mathcal O_\lambda}$ and
$\overline{\mathcal O_\lambda}$ is $\operatorname{SO}(V)$-stable, we
have
\[
\mathcal P_r\subseteq\overline{\mathcal O_\lambda}.
\]
Moreover,
$\dim\mathcal P_r=2r(N-2r-1)$ \cite{KP}. Therefore
\[
\begin{aligned}
\frac12\dim\mathcal O_\lambda
&=\dim\bigl(\overline{\mathcal O_\lambda}
\cap\mathfrak g^{\rho(\Gm)}\bigr)\\
&=r(2p-2r-1)\\
&\leq r(N-2r-1)
=\frac12\dim\mathcal P_r\\
&\leq\frac12\dim\mathcal O_\lambda.
\end{aligned}
\]
The last inequality follows from
$\mathcal P_r\subseteq\overline{\mathcal O_\lambda}$. Thus all the
inequalities are equalities. Since $r>0$, we have $2p=N$, so $p=n$.
Moreover,
$\dim\mathcal P_r=\dim\mathcal O_\lambda$. Since
$\mathcal P_r\subseteq\overline{\mathcal O_\lambda}$, we obtain
$\mathcal P_r=\mathcal O_\lambda$.

We next show that $r<m$. Suppose that $r=m$. Then $R_r$ is dense in
$\Lambda^2W$. Since
$R_r\subseteq\mathcal P_r=\mathcal O_\lambda$ and
$\overline{\mathcal O_\lambda}$ is closed, it follows that
$\Lambda^2W\subseteq\overline{\mathcal O_\lambda}$. Hence
\[
\dim\bigl(\overline{\mathcal O_\lambda}
\cap\mathfrak g^{\rho(\Gm)}\bigr)
=\dim\mathfrak g^{\rho(\Gm)},
\]
whereas \eqref{eq:fixed-dimensions} says that the right-hand side is
at least one larger than the left-hand side. This is a contradiction.
Therefore
\[
1\leq r<m=\lfloor n/2\rfloor,
\qquad
\lambda=\lambda_r=(2^{2r},1^{N-4r}).
\]
Since $r$ is an integer, the first inequality is equivalent to
$1\leq r\leq\lfloor(n-2)/2\rfloor$.
This proves the assertion when $p\geq2$.

Suppose now that $p=1$. Let $b<a_{\max}$ be the largest weight with
$V_b\ne0$. Such a weight exists because $N\geq8$. We claim that
$b\geq0$. Suppose that $b<0$. Then $-b$ is also a weight. Since
$a_{\max}$ is maximal, $-b\leq a_{\max}$. The choice of $b$ excludes
$b<-b<a_{\max}$. Hence $-b=a_{\max}$ and $b=-a_{\max}$.
There is no smaller weight, since its negative would be larger than
$a_{\max}$. Thus these are the only two weights. Since
$\dim V_{a_{\max}}=\dim V_{-a_{\max}}=1$, we obtain $N=2$, a
contradiction. Therefore $b\geq0$.

Since $\Lambda^2W=0$, the largest
$\operatorname{Ad}(\widetilde\mu)$-weight on $\Lambda^2V$ is
$a_{\max}+b$. Indeed, a nonzero wedge involving $W$ has weight at
most $a_{\max}+b$, while a wedge not involving $W$ has weight at most
$2b<a_{\max}+b$. The weight $a_{\max}+b$ occurs only on
$W\otimes V_b$. Therefore
\[
c=a_{\max}+b,
\qquad
\mathfrak g^{\rho(\Gm)}=W\otimes V_b.
\]

We next show that $b=0$. Suppose that $b>0$. Fix $0\ne w\in W$.
Every nonzero element of $W\otimes V_b$ has the form $w\wedge u$
with $0\ne u\in V_b$.
The vectors $w$ and $u$ are isotropic and orthogonal. Hence
$w\wedge u$ is square-zero of rank two and belongs to
$\mathcal O_{\min}$. Since every nonzero nilpotent orbit closure
contains $\overline{\mathcal O}_{\min}$, we obtain
\[
\mathfrak g^{\rho(\Gm)}
\subseteq\overline{\mathcal O}_{\min}
\subseteq\overline{\mathcal O_\lambda},
\]
so the intersection in \eqref{eq:fixed-dimensions} equals
$\mathfrak g^{\rho(\Gm)}$ on closed points. This is contrary to the
second inequality in \eqref{eq:fixed-dimensions}. Thus $b=0$.

Put $U=V_0$. Any negative weight other than $-a_{\max}$ would
have a positive opposite weight strictly between $0$ and $a_{\max}$.
Thus $U$ is nondegenerate, $\dim U=N-2$, and
\[
\mathfrak g^{\rho(\Gm)}=W\otimes U.
\]
Fix $0\ne w\in W$. The map $u\mapsto w\wedge u$ is a linear
isomorphism from $U$ to $W\otimes U$. For $A=w\wedge u$ we have
\[
A^2(x)=-(u,u)(w,x)w.
\]
If $u\ne0$, the vectors $w$ and $u$ are linearly independent. The
two linear forms $(w,-)$ and $(u,-)$ are also linearly independent,
so the image of $A$ is $\langle w,u\rangle$. Thus $A$ has rank two.
The formula above shows that its Jordan partition is
\[
\begin{cases}
(2^2,1^{N-4}),& (u,u)=0,\\
(3,1^{N-3}),& (u,u)\ne0.
\end{cases}
\]
Indeed, in the first case $A^2=0$, so there are two Jordan blocks of
size two. In the second case, $A^2$ has rank one and $A^3=0$, so there
is one Jordan block of size three.

We claim that every closed point $w\wedge u$ of
$\overline{\mathcal O_\lambda}\cap\mathfrak g^{\rho(\Gm)}$
satisfies $(u,u)=0$. Suppose that one such point has $(u,u)\ne0$.
The partition $(3,1^{N-3})$ is not very even, so all elements
$w\wedge u'$ with $(u',u')\ne0$ lie in the same
$\operatorname{SO}(V)$-orbit. They would therefore all belong to
$\overline{\mathcal O_\lambda}$. The locus
$\{u'\in U:(u',u')\ne0\}$ is dense in $U$. Since
$\overline{\mathcal O_\lambda}$ is closed, it would contain
$W\otimes U=\mathfrak g^{\rho(\Gm)}$, again contradicting
\eqref{eq:fixed-dimensions}. This proves the claim.

Conversely, $0\in\overline{\mathcal O_\lambda}$, and every nonzero
$w\wedge u$ with $(u,u)=0$ belongs to $\mathcal O_{\min}$.
Since
$\overline{\mathcal O}_{\min}\subseteq
\overline{\mathcal O_\lambda}$, the closed points of
$\overline{\mathcal O_\lambda}\cap\mathfrak g^{\rho(\Gm)}$ are
exactly
\[
\{w\wedge u:(u,u)=0\}.
\]
The form on $U$ is nondegenerate and $\dim U=N-2$, so this quadratic
cone has dimension $N-3$. By \eqref{eq:fixed-dimensions}, we have
\[
\frac12\dim\mathcal O_\lambda=N-3
=\frac12\dim\mathcal O_{\min}.
\]
Since $\mathcal O_{\min}\subseteq\overline{\mathcal O_\lambda}$ and
the two orbits have the same dimension, we have
\[
\lambda=(2^2,1^{N-4}),
\qquad
r=1\leq\lfloor(n-2)/2\rfloor.\qedhere
\]
\end{proof}

\subsection{Proof of the main theorem}\label{sec:mainproof}

\begin{proof}[Proof of Theorem \ref{thm:classification}]
Suppose that
\[
\overline{\mathcal O}\cong(T^*X)^{\aff}
\]
for a smooth quasi-affine variety $X$.
Fiber dilation induces a nontrivial representation $\rho$ on
$\mathfrak g$. By Section \ref{sec:fiber}, $\rho$ has only nonnegative
weights.
Propositions \ref{prop:typeA}, \ref{prop:typeB}, \ref{prop:typeC},
and \ref{prop:typeD} prove the necessity.
We now recall the constructions of Fu and Liu \cite{FuLiu}, which
realize every case in the first table and every example in the second
table.
Let $P$ be a cominuscule parabolic with Levi decomposition
$\mathfrak p=\mathfrak l\oplus\mathfrak u^+$, and set
$X=\mathcal O\cap\mathfrak u^+$.
For the cases in \cite[Table 1]{FuLiu}, the variety $X$ is a smooth
$L$-orbit. By \cite[Theorem 4.2]{FuLiu}, we have the following
isomorphism of affine varieties:
\[
(T^*X)^{\aff}\cong\overline{\mathcal O}.
\]

In type $A_{n-1}$, fix $p+q=n$ with $r<p\leq q$ and take $P=P_p$.
Then $\mathfrak u^+\cong\Mat_{p\times q}$, and $X$ is the locus of
matrices of rank $r$.
In type $B_n$ with $n\geq3$, take $P=P_1$. Then
$\mathfrak u^+\cong\C^{2n-1}$, and $X$ is the punctured isotropic cone
\[
\{0\ne v\in\C^{2n-1}:(v,v)=0\}.
\]
In type $C_n$, take $P=P_n$. Then
$\mathfrak u^+\cong\Sym^2\C^n$, and $X$ is the locus of symmetric
matrices of rank $r$. In type $D_n$, take $P=P_n$. Then
$\mathfrak u^+\cong\Lambda^2\C^n$, and for $r\geq2$, the variety $X$
is the locus of skew-symmetric matrices of rank $2r$.

For $r=1$ in type $D_n$, apply \cite[Theorem 1.4]{FuLiu} to the
irreducible Hermitian symmetric space $D_n/P_n$.
The corresponding highest weight variety is
$\operatorname{Gr}(2,n)\subset\PP(\Lambda^2\C^n)$.
Its punctured affine cone is the rank-two locus in $\Lambda^2\C^n$.
By \cite[Theorem 1.4]{FuLiu}, this is the required cotangent model.
There is also the quadratic construction from
\cite[Table 1 and Theorem 4.2]{FuLiu}: take $P=P_1$, so that
$\mathfrak u^+\cong\C^{2n-2}$, and let
\[
X=\{0\ne u\in\C^{2n-2}:(u,u)=0\}.\qedhere
\]
\end{proof}

\section{Exceptional types.}

\subsection{Nonexistence in types \texorpdfstring{$G_2$, $F_4$, and $E_8$.}{G2, F4, and E8}}\label{sec:G2F4E8}

In this section we upgrade the main theorem of \cite{Jia} to all
nonzero nilpotent orbits in types $G_2$, $F_4$, and $E_8$.

\begin{theorem}\label{thm:negative-exceptional}
Let $\mathfrak g$ be of type $G_2$, $F_4$, or $E_8$, and let
$\mathcal O\subset\mathfrak g$ be a nonzero nilpotent orbit.
There is no smooth quasi-affine variety $X$ such that
\[
\overline{\mathcal O}\cong(T^*X)^{\aff}.
\]
\end{theorem}

\begin{proof}
Suppose that such an $X$ exists, and let
$\rho:\Gm\to\operatorname{GL}(\mathfrak g)$ be the tangent
representation induced by fiber dilation. By Section~\ref{sec:fiber},
$\rho$ is nontrivial, all of its weights are nonnegative, and
$\rho(\Gm)$ preserves
$\overline{\mathcal O}$. Lemma \ref{lem:minimal-preserved} then shows
that $\rho(\Gm)$ preserves $\overline{\mathcal O}_{\min}$.
By Proposition \ref{prop:projective-stabilizer}, its projective image
lies in the adjoint group $G$.
By Lemma \ref{lem:top-root}, we have
$\mathfrak g^{\rho(\Gm)}=\mathfrak g_I$ for some nonempty $I$.

As computed in \cite[Proposition 3.1]{Jia},
\[
\max_i\dim\mathfrak g_{\{i\}}
=
2,\ 7,\ 14
\]
in types $G_2$, $F_4$, and $E_8$, respectively.
Since $\mathfrak g_I\subseteq\mathfrak g_{\{i\}}$ for every $i\in I$,
we obtain $\dim\mathfrak g^{\rho(\Gm)}\leq2,7,14$ in the three types.
By Lemma \ref{lem:strict-bound}, on the other hand,
$\dim\mathfrak g^{\rho(\Gm)}\geq\frac12\dim\mathcal O+1$.
The dimensions of the minimal orbits are respectively $6$, $16$, and $58$
\cite[Section 8.4]{CM}. Hence for every nonzero orbit the right-hand
side is at least $4$, $9$, and $30$, respectively.
This is a contradiction.
\end{proof}

\subsection{Types \texorpdfstring{$E_6$ and $E_7$}{E6 and E7}}\label{sec:E6E7}

The minimal nilpotent orbit closures in types $E_6$ and $E_7$ admit
explicit cotangent models.
By \cite[Theorem 1.4 and Table 2]{FuLiu}, there are smooth quasi-affine
varieties $X$ such that
\[
\overline{\mathcal O}_{\min}\cong(T^*X)^{\aff}.
\]
In type $E_6$, one may take the punctured spinor cone.
In type $E_7$, one may take the punctured Cayley cone.

\begin{theorem}\label{thm:E6E7}
Let $\mathfrak g$ be of type $E_6$ or $E_7$, and let $\mathcal O$ be a
nonzero nilpotent orbit. If
\[
\overline{\mathcal O}\cong(T^*X)^{\aff}
\]
for a smooth quasi-affine variety $X$, then $\mathcal O$ is one of
\[
\mathcal O_{\min}(E_6),\qquad
\mathcal O_{\min}(E_7),\qquad
\mathcal O(2A_1)\subset\mathfrak e_7.
\]
\end{theorem}

\begin{proof}
Suppose that
\[
\overline{\mathcal O}\cong(T^*X)^{\aff}
\]
for a smooth quasi-affine variety $X$. Let
$\rho:\Gm\to\operatorname{GL}(\mathfrak g)$ be the tangent
representation induced by fiber dilation. By Lemmas
\ref{lem:fixed-tangent-cone} and \ref{lem:strict-bound},
\begin{equation}\label{eq:exceptional-lower-bound}
\dim\mathfrak g^{\rho(\Gm)}
\geq\frac12\dim\mathcal O+1.
\end{equation}

By \cite[Lemma 5.2]{Jia}, the representation $\rho(\Gm)$ preserves the
tangent cone $C_0\overline{\mathcal O}=\overline{\mathcal O}$. Hence it
preserves $\overline{\mathcal O}_{\min}$ by Lemma
\ref{lem:minimal-preserved}. Its projective image lies in the identity
component of the projective stabilizer of the adjoint variety, which
is the adjoint group $G$ by Proposition
\ref{prop:projective-stabilizer}.
Thus, by Lemma \ref{lem:top-root}, up to conjugation by $G$, there
is a nonempty set $I$ of simple roots such that
\[
\mathfrak g^{\rho(\Gm)}=\mathfrak g_I.
\]
For every $i\in I$, we have
$\mathfrak g_I\subseteq\mathfrak g_{\{i\}}$.

We use the Bourbaki numbering of the simple roots. A direct enumeration
of the positive roots gives
\[
\begin{array}{c|rrrrrrr}
 &1&2&3&4&5&6&7\\
\hline
E_6&16&1&5&2&5&16&-\\
E_7&1&7&2&3&5&10&27
\end{array}
\qquad
\left(\dim\mathfrak g_{\{i\}}\right)_{i}.
\]
Indeed, $\mathfrak g_{\{i\}}$ is the sum of the one-dimensional root
spaces for the roots whose $\alpha_i$-coefficient is the coefficient
of $\alpha_i$ in the highest root. Thus
\begin{equation}\label{eq:E6E7-upper-bound}
\dim\mathfrak g^{\rho(\Gm)}\leq16
\quad\text{in type }E_6,
\qquad
\dim\mathfrak g^{\rho(\Gm)}\leq27
\quad\text{in type }E_7.
\end{equation}

Suppose first that $\mathfrak g$ is of type $E_6$. The minimal orbit
has dimension $22$. Every other nonzero nilpotent orbit has dimension
at least $32$ \cite[Section 8.4]{CM}. If $\mathcal O$ is not minimal,
then by \eqref{eq:exceptional-lower-bound},
\[
\dim\mathfrak g^{\rho(\Gm)}\geq17,
\]
contrary to \eqref{eq:E6E7-upper-bound}. Thus
$\mathcal O=\mathcal O_{\min}(E_6)$.

Now suppose that $\mathfrak g$ is of type $E_7$. The minimal orbit has
dimension $34$, and $\mathcal O(2A_1)$ has dimension $52$. Every other
nonzero nilpotent orbit has dimension at least $54$
\cite[Section 8.4]{CM}. If $\mathcal O$ is neither of these two
orbits, then by \eqref{eq:exceptional-lower-bound},
\[
\dim\mathfrak g^{\rho(\Gm)}\geq28,
\]
again contrary to \eqref{eq:E6E7-upper-bound}. Therefore
\[
\mathcal O=\mathcal O_{\min}(E_7)
\quad\text{or}\quad
\mathcal O=\mathcal O(2A_1).
\qedhere
\]
\end{proof}

\bigskip
\noindent
Boming Jia

\noindent
Yau Mathematical Sciences Center,\\
Jingzhai 301, Tsinghua University,\\
Beijing 100084, China

\noindent
Email: \href{mailto:jiabm@tsinghua.edu.cn}{jiabm@tsinghua.edu.cn}

\end{document}